\documentclass[a4paper]{article}

\usepackage{graphicx}
\usepackage[utf8]{inputenc}
\usepackage{amsmath, amsthm, amssymb, amsfonts,mathtools}

\usepackage{fourier}
\usepackage[french,english]{babel}
\usepackage{microtype}
\usepackage[backend=bibtex,style=alphabetic]{biblatex}

\newtheorem{theorem}{Theorem}[section]
\newtheorem{lemma}[theorem]{Lemma}
\newtheorem{proposition}[theorem]{Proposition}

\newtheorem{definition}[theorem]{Definition}

\newtheorem*{proposition*}{Proposition}
\newtheorem*{lemma*}{Lemma}
\newtheorem*{theorem*}{Theorem}
\newtheorem*{corollary*}{Corollary}

\title{Closed ray nil-affine manifolds and parabolic geometries}

\author{Raphaël V.~{\sc Alexandre}\footnote{Institut de Math\'ematiques de Jussieu-Paris Rive Gauche, Sorbonne Université, 4 place Jussieu, 75252 Paris Cédex, France.
ACG,
OURAGAN (IMJ-PRG, INRIA Paris, Sorbonne Université, Université de Paris, CNRS).
Email address: {\tt raphael.alexandre@math.cnrs.fr}.}}

\newcommand\cN{\mathcal{N}}

\newcommand\dd{\,{\rm d}}

\DeclareMathOperator*\id{id}
\newcommand\R{\mathbf{R}}
\newcommand\C{\mathbf{C}}
\newcommand\CP{\mathbf{CP}}

\newcommand\Heis{{\rm Heis}}
\DeclareMathOperator*\Aff{Aff}
\DeclareMathOperator\ad{ad}
\DeclareMathOperator\Ad{Ad}
\DeclareMathOperator*\GL{GL}
\DeclareMathOperator*\SU{SU}

\DeclareMathOperator*\SL{SL}

\newcommand\vol{{\rm vol}}
\DeclareMathOperator*\Aut{Aut}
\newcommand\ii{\boldsymbol{i}}
\newcommand{\iI}{\mathopen{[}0\,,1\mathclose{]}}
\newcommand\rT{{\rm T}}
\renewcommand\Re{{\rm Re}}

\newcommand\thmpartialcomp{
Let $\cN$ be a commutative or two-step nilpotent space.
Let $(G_1,\cN)$ be a ray geometry of rank one.  Let $M$ be
a closed $(G_1,\cN)$-manifold.
Then $M$ is either complete or there exists a nil-affine subspace $I\subset \cN$ such that $D\colon\widetilde M \to \cN-I$ is a cover onto its image.
}
\newcommand\thmauto{
Let $\cN$ be a commutative or two-step nilpotent space.
Let $(G_1,\cN)$ be a   ray geometry of rank one.  Let $M$ be
a closed $(G_1,\cN)$-manifold. If $\Aut(M)$ does not act properly on $M$ then $M$ is complete.
}

\newcommand\thmmarkus{
Let $\cN$ be a commutative or two-step nilpotent space.
Let $(G_1,\cN)$ be a ray geometry of rank one with parallel volume. Every closed $(G_1,\cN)$-manifold   is complete.
}

\bibliography{ref.bib}

\begin{document}
\maketitle

\begin{abstract}
Ray nil-affine geometries are defined on nilpotent spaces. They occur in every parabolic geometry and in those cases, the nilpotent space is an open dense subset of the corresponding flag manifold. 

We are interested in closed manifolds having a ray nil-affine structure. We show that under a rank one condition on the isotropy, closed manifolds are either complete or their developing map is a cover onto the complement of a nil-affine subspace. We prove that if additionally there is a parallel volume or if the automorphism group acts non properly then closed manifolds are always complete.

This paper is a sequel to a previous work on ray manifolds in affine geometry.
\end{abstract}

\section{Introduction}

Nilpotent spaces and nil-affine geometries are the natural generalization of  Euclidean spaces and  affine geometries. Interestingly, the study of nil-affine  geometries has an importance in  parabolic geometries.

If one has a non-compact symmetric space $G/K$, its visual boundary is generally not homogeneous but has several flag manifolds $G/P$ as boundary components. The choice of a parabolic geometry $(G,G/P)$ corresponds to the choice of a parabolic subgroup $P$. A parabolic subgroup is a closed subgroup of $G$ containing a minimal parabolic subgroup (a conjugate of the Borel subgroup $B\subset G$).

For instance,  the conformal geometry is the (only) parabolic geometry of the hyperbolic space.
More generally, the Furstenberg boundary is the parabolic manifold for the choice of the Borel subgroup as isotropy.

Once a geometry $(G,G/P)$ is chosen, one can try to classify manifolds with a geometric structure modeled on $(G,G/P)$ (it is the choice of an atlas with values in $G/P$ and transitional maps in $G$). One of the most general questions is:

\emph{What are the closed manifolds that can be modeled on the geometry $(G,G/P)$ ?} 
This question is very open. In conformal geometry it is not known what are the closed manifolds modeled on it.

\vskip10pt
To investigate this question there are several possibilities. One can reduce $G/P$ to $\Omega\subset G/P$ in order to obtain an invariant metric on $\Omega$ (invariant for the subgroup preserving $\Omega$).  For example the hyperbolic disc can be identified to a disc $\Omega\subset S^2$ in the conformal geometry of the sphere. In this direction, Zimmer~\cite{Zimmer} studies closed manifolds with developing maps  for which the image is \emph{bounded in an affine chart} $\Omega\subset G/P$.

Another direction that we choose is
to reduce $G/P$ to an open dense subset.  The open dense subset we choose  is known as an open Schubert stratum in~\cite{Kapovich} (see also~\cite{Guivarch}). This open dense subset is a nilpotent space $N$ that can be seen as a subgroup of $P$. So $(P,N)$ is a subgeometry of $(G,G/P)$. The isotropy $L$ of $(P,N)$ is the \emph{Levi} subgroup $L\subset P$ and we call $(P,N)=(P,P/L)$ a \emph{Levi geometry}.

\vskip10pt
In rank one parabolic geometry (the parabolic geometries of the real rank one non-compact symmetric spaces), the study of this geometry $(P,P/L)$ has major consequences on the geometry $(G,G/P)$. For instance: if a closed manifold modeled on $G/P$ avoids a point, then its developing map is a cover onto its image.
We expect the study of $(P,P/L)$ to be again fruitful in higher rank parabolic geometries.

In rank one parabolic geometry, $P=N\rtimes L$  with $L=MA$ the combination of an abelian group $A$ (that has dimension one) and a rotational centralizing group $M\subset K$. The study of $P/L$ for rank one parabolic geometries has been achieved through the work of Fried~\cite{Fried}, Miner~\cite{Miner} and more recently by the author~\cite{Ale}.

Those are the first examples of what we will call \emph{ray nil-affine} geometries.
Those geometries appear in very different contexts, and in particular in parabolic geometries of higher rank. They also have an interest in affine geometry and were studied in a previous work of the author~\cite{Ale3}.

A \emph{ray geometry} $(N\rtimes KA,N)$ on a nilpotent space $N$ is the data of a linear group $KA$ such that $A$ is an abelian subgroup and $K$ centralizes $A$ and is a compact subgroup. The \emph{rank} of a ray geometry is the dimension of $A$.
In a higher rank parabolic geometry, one can find different ray nil-affine geometries. They are  more constrained than a Levi geometry $(P,P/L)$. However, we intend to show that they are rich examples of parabolic subgeometries that can also appear in other contexts.

We will give a generic construction (see proposition~\ref{prop-pararay}) to obtain ray geometries $(N\rtimes KA,N)\subset (P,N)$ in any Levi geometry.

\vskip10pt
In this paper, we address the question of the geometry of closed manifolds having a ray nil-affine structure.
We extend the work in affine geometry~\cite{Ale3} to a setting where the space can be nilpotent of order two. We believe those results should be true for any nilpotency order.

We define nil-affine subspaces to be the left translations of linear subspaces of the nilpotent Lie algebra. For instance, geodesic curves are (in this sense) left translations of linear line segments. See section~\ref{sec-2}.

\begin{theorem*}[\ref{thm-partialcomp}]
\thmpartialcomp
\end{theorem*}

\begin{theorem*}[\ref{thm-markus}]
\thmmarkus
\end{theorem*}

\begin{theorem*}[\ref{thm-auto}]
\thmauto
\end{theorem*}

\paragraph{Organization of the paper}In section~\ref{sec-2} we present the general setting of ray geometries. In section~\ref{sec-3} we construct explicitly such geometries as subgeometries living in a flag manifold. This construction is general to any parabolic geometry. In section~\ref{sec-4} we explain how to derive the proof of theorem~\ref{thm-partialcomp} from~\cite{Ale3} and additional arguments related to the nilpotency. In section~\ref{sec-5} we study the reducibility of ray manifolds and prove theorems~\ref{thm-markus} and~\ref{thm-auto}.

\paragraph{Acknowledgment}This work is part of the author's doctoral thesis, under the supervision of Elisha Falbel. The author also enjoyed many conversations with Martin Mion-Mouton and Charles Frances on parabolic geometries.

\section{Nilpotent spaces with ray geometries}\label{sec-2}

Let $N$ be a simply connected nilpotent Lie group. Denote $\mathfrak n$ its Lie algebra.
The correspondance $\exp\colon\mathfrak n\to N$ is a diffeomorphism and we have the Baker-Campbell-Hausdorff formula:
 \begin{equation}
\exp(x)\exp(y) = \exp\left(x+y+\frac 12 [x,y] + \frac 1{12}([x,[x,y]] - [y,[x,y]])+\dots\right).
\end{equation}

\begin{definition}
A  \emph{nil-affine subspace} of $N$ is a left translation $L_x\exp(V)$ of the exponential in $N$ of a linear subspace $V\subset\mathfrak n$.
\end{definition}

In particular when $V$ is a line of $\mathfrak n$ we can deduce a notion of geodesic curve.

\begin{definition}
A \emph{geodesic segment}  in $N$ is the left-translation of a curve $\exp(tv)$ for $v\in\mathfrak{n}$ (the Lie algebra of $N$) and $t\in \iI$.
\end{definition}

\begin{lemma}
When $N$ is at most two-step nilpotent, the geodesic convex subsets are the affine convex subsets.
\end{lemma}
\begin{proof} 
The Baker-Campbell-Hausdorff formula shows that if $\exp(a+tv)$ is an affine segment then any left-translation $\exp(x)\exp(a+tv) = \exp(x+a+tv + 1/2 [x,a+tv]) = \exp(x+a+ 1/2 [x,a] + t(v + 1/2 [x,v]))$ is again an affine segment.
\end{proof}

\begin{definition}
Let $D\colon \widetilde M \to N$ be a local diffeomorphism. A curve $\gamma\colon \iI \to \widetilde M$ is \emph{geodesic} if $D(\gamma)$ is a geodesic segment in $N$.

For any $p\in \widetilde M$, define $V_p\subset \rT_p\widetilde M$  the set of the vectors such that there exists a geodesic segment $\gamma\colon \iI \to \widetilde M$ with $\gamma(0)=p$, $\gamma'(0)\in V_p$.  We say that $V_p$ is the \emph{visibility set} from (or of) $p\in \widetilde M$.
\end{definition}

\begin{definition}
Let $D\colon\widetilde M \to N$ be a local diffeomorphism. A subset $C\subset \widetilde M$ is \emph{convex} if $D$ is injective on $C$ and $D(C)$ is convex. If $x\in\widetilde M$ and $y\in N$, we say that $y$ \emph{is visible from} $x$ if there exists a geodesic segment from $x$ to $z\in\widetilde M$ such that $D(z)=y$.
\end{definition}

\begin{lemma}
The visible set $V_p\subset \rT_p\widetilde M$ is open.
\end{lemma}
\begin{proof}
It is necessarily a neighborhood of $0$ since $D$ is a local diffeomorphism and $N$ is locally convex.

Consider $v \in V_p$ and its associated geodesic segment $\gamma$. Cover $\gamma$ by finitely many convex opens $U_i$. Then $D(\gamma)$ can be continuously deformed into geodesic segments in $\bigcup D(U_i)$. They correspond to geodesic segments of $\widetilde M$ if they are always entirely contained in $\bigcup D(U_i)$. Since the $U_i$ are in finite number, it can be done for an open set around $v\in \rT_p\widetilde M$.
\end{proof}

Following the convexity arguments of Carrière~\cite{Carriere} (see also~\cite{Koszul,Benzecri}) we obtain:

\begin{proposition}[\cite{Ale}]
Let $D\colon\widetilde M \to N$ be a local diffeomorphism.
\begin{itemize}
\item
If $C_1,C_2$ are convex subsets in $\widetilde M$ and $C_1\cap C_2\neq \emptyset$, then $D$ is injective on $C_1\cup C_2$.
\item 
Let $C$ be convex and containing $p\in \widetilde M$. Then $C\subset \exp_p(V_p)$.
\item $D$ is injective on $\exp_p(V_p)$ for any $p\in \widetilde M$.
\item If for $p\in \widetilde M$, $\exp_p(V_p)$ is convex then $\exp_p(V_p)=\widetilde M$.
\item $D$ is  a diffeomorphism if and only if for any $p\in\widetilde M$, $V_p= \rT_p\widetilde M$.
\end{itemize}
\end{proposition}

\begin{definition}
A manifold $M$ is a $(G,X)$-manifold for $G$ a group acting on a space $X$ if there exists $(D,\rho)$ a pair of a local diffeomorphism $D\colon\widetilde M \to X$ (called the developing map) and a morphism $\rho\colon\pi_1(M)\to G$ (called the holonomy morphism) such that
\begin{equation}
\forall \gamma\in\pi_1(M)\forall x\in\widetilde M, \; D(\gamma\cdot x)=\rho(\gamma)D(x).
\end{equation}

Let $X=\cN$ be a nilpotent Lie group and $\GL_{\Aut}(\mathfrak n)$ be its largest linear automorphism group. Let $\Aff(N) = N\rtimes \GL_{\Aut}(\mathfrak n)$. We say that $(\Aff(N),\cN)$ is the \emph{nil-affine geometry} of $\cN$ and if $M$ has a $(G,\cN)$-structure then it is  a  \emph{nil-affine manifold}.
\end{definition}

\begin{definition}
Let $M$ be a nil-affine manifold. It is \emph{complete} if $D\colon\widetilde M \to \cN$ is a diffeomorphism.
\end{definition}

\vskip10pt
They are several ways to describe ray geometries on a nilpotent space. We give here a simple axiomatic description.

\begin{definition} \label{def-ray}
Let $N$ be a simply connected nilpotent Lie group with Lie algebra $\mathfrak n$.
We assume that $\mathfrak n$ is graded:
\begin{equation}
\mathfrak n = \mathfrak n_1\oplus\dots\oplus\mathfrak n_k
\end{equation}
with $[\mathfrak n_i,\mathfrak n_j]\subset \mathfrak n_{i+j}$.

A \emph{ray geometry} on the space $\cN=N$ is a pair $(G,\cN)$ with $G= N\rtimes KA$ and $KA\subset \GL_{\Aut}(\mathfrak n)$  such that the following conditions are verified.
\begin{enumerate}
\item The subgroup $A$ is isomorphic to a multiplicative group $(\R_+^*)^r$ with $r=\dim A$ called the \emph{rank of the ray geometry}. The subgroup $K$  is compact and centralizes $A$.
\item There exists a basis $(e_1,\dots,e_n)$ such that each $e_i$ belongs to a single $\mathfrak n_j$.
And in respect to the basis $(e_1,\dots,e_n)$, there exists a linear isomorphism $(\alpha_1,\dots,\alpha_r)\colon\mathfrak a \to \R^r$ and real numbers $d_{i,j}$  such that for any $\exp(a)\in A\subset \GL_{\Aut}(\mathfrak n)$:
\begin{equation}
\exp(a)=
\begin{pmatrix}
\exp(\alpha_1(a))^{d_{1,1}}\cdots\exp(\alpha_r(a))^{d_{r,1}} \\ 
&\ddots \\
 &&\exp(\alpha_1(a))^{d_{1,n}}\cdots\exp(\alpha_r(a))^{d_{r,n}}\label{eq-dilatationmatrix}
\end{pmatrix}
.
\end{equation}
\end{enumerate}
\end{definition}

The second hypothesis on  the basis $(e_1,\dots,e_n)$ allows to apply the following result. In general, one could ask under which condition a linear decomposition $\mathfrak n = L_1\oplus L_2$ of any nilpotent Lie algebra gives a decomposition $N=\exp(L_1)\exp(L_2)$. In Bourbaki~\cite[Prop. 17, p. 237]{Bourbaki} it is shown that:

\begin{lemma}[\cite{Bourbaki}]\label{lem-decomp}
Let $\mathfrak n$ be any nilpotent Lie algebra.
Consider $\mathfrak n=\mathfrak n^0\supset \mathfrak n^1 \supset\dots\supset \mathfrak n^m = \{0\}$ any sequence of ideals of $\mathfrak n$ such that $[\mathfrak n,\mathfrak n^i]\subset \mathfrak n^{i+1}$.

Consider a  linear decomposition $\mathfrak n = L_1\oplus L_2$. 
If for every $\mathfrak n^i$ we have $\mathfrak n^i = (\mathfrak n^i\cap L_1)\oplus (\mathfrak n^i\cap L_2)$ then 
for any $x\in N$ there exists a unique decomposition 
\begin{equation}
x = \exp(x_1) \exp(x_2)
\end{equation}
with  $x_1\in L_1$ and $x_2\in L_2$.
\end{lemma}

The idea of the proof is to proceed by induction on the dimension, by taking the quotient by the centralizer. (On a commutative space we always have the property.)

In particular, in our context, one can consider 
\begin{equation}
\mathfrak n^i = \mathfrak n_i \oplus \mathfrak n_{i+1}\oplus \dots\oplus \mathfrak n_k
\end{equation}
as sequence of decreasing ideals of $\mathfrak n$. The hypothesis $\mathfrak n^i = (\mathfrak n^i\cap L_1)\oplus (\mathfrak n^i\cap L_2)$ is necessarily verified if
\begin{equation}
\mathfrak n_i = (\mathfrak n_i\cap L_1)\oplus (\mathfrak n_i\cap L_2).
\end{equation}
So if we were to consider $L_1$ and $L_2$ to be subspaces generated by subfamilies of $(e_1,\dots,e_n)$, the property would be verified.

\vskip10pt
A related consideration that is important in $\mathfrak n$ is:
\begin{equation}
\Ad_{\exp(x)} = \exp(\ad_x).
\end{equation}
In consequence, we have in $N$:
\begin{align}
\exp(x) \exp(y) &= \exp\left(\sum_{n\geq 0} \frac{1}{n!}\ad_x^n(y) \right)\exp(x)\\
&=\exp\left( y + [x,y] + \frac{1}{2}[x,[x,y]] + \frac{1}{3!}[x,[x,[x,y]]] + \dots \right)\exp(x).
\end{align}

\paragraph{Limit sets}
Later  we will need to study sequences of subsets that have a limit.

\begin{definition}
Let $B_n$ be a sequence of  subsets. Its \emph{limit set}, denoted $\lim B_n$, is the set of the points $\lim x_n$ for sequences $x_n\in B_n$ (with $x_n$ chosen for each $B_n$).
\end{definition}

\begin{proposition}\label{prop-visconv}
Let $g_{i}\in \pi_1(M)$ be a sequence of transformations and $S\subset \widetilde M$ be closed and convex. Assume that $g_{i}S$ has a limit point $y\in \widetilde M$. Consider $B_\infty = \lim D(g_{i}S) = \lim \rho(g_i) D(S)$ the limit set in the developing map.  Then $B_\infty$ is closed, convex and there exists a closed and  convex subset $S_\infty$  containing $y$ such that $D(S_\infty)=B_\infty$.
\end{proposition}

\begin{proof}
Observe that by convexity, each $D(g_i S)$ is closed and convex. Therefore $B_\infty$ is again closed.
It is also convex. If two points $x_1,x_2$ belong to the limit set $B_\infty$ then we consider their associated converging sequences. We can form geodesics in each $g_i S$ that joins the points from each sequence. The limit of those geodesics exists, is geodesic and connects $x_1$ and $x_2$.

Similarly, $S_\infty=\lim g_i S$ is closed and convex.
We show its developing image covers $B_\infty$. Let $y_i\in g_i S$ be a sequence tending to $y$. Let $z\in B_\infty$ be the limit of $D(z_i)$ for $z_i\in g_i S$. The geodesics $\gamma_i$ from $y_i$ to $z_i$ are developed into $D(\gamma_i)$ and those tend to the geodesic from $D(y)$ to $D(z)$. This geodesic is compact (defined on $\iI$) and each $\gamma_i$ is completely visible. Therefore $\lim \gamma_i$ corresponds to a visible geodesic issued from $y$. It necessarily ends at $z$ by injectivity.
\end{proof}

\paragraph{Example 1}Let $\Heis$ be the three-dimensional Heisenberg group, with the Lie bracket $[X,Y]=Z$. Let $e_1$ be the direction of $X$, $e_2$ of $Y$ and $e_3$ of $Z$. Then the maximal group $A$ acting on $\Heis$ is
\begin{equation}
\begin{pmatrix}
\beta_1  \\
& \beta_2 \\
&&\beta_1\beta_2
\end{pmatrix}
\end{equation}
since we necessarily have $[\beta_1 X,\beta_2 Y] = \beta_1\beta_2 Z$. This ray geometry has many ray subgeometries of rank one by deciding any relation between $\beta_1$ and $\beta_2$. For example, the similarity geometry (as in~\cite{Ale}) is given by deciding $\beta_2=\beta_1$. The volume-preserving geometry is given by deciding $\beta_2 = \beta_1^{-1}$.

\paragraph{Example 2}Any semisimple Lie groupe $G$ has an Iwasawa decomposition $G= KAN$. Let $M$ be the subgroup of $K$ centralizing $A$, then the Borel subgroup $B=MAN$ is a ray geometry on $N$. 
Indeed, with a root space decomposition of $\mathfrak g$:
\begin{equation}
\mathfrak g = \mathfrak g_0 \oplus \bigoplus_{\alpha\in \Phi}g_\alpha
\end{equation}
the Lie algebra
 $\mathfrak n = \mathfrak n_+$ consisting of the subspaces corresponding to positive roots $\Phi_+\subset \Phi$  is  a graded nilpotent Lie algebra on which $\mathfrak m\oplus\mathfrak a = \mathfrak g_0$ acts and preserves each $\mathfrak g_\alpha$. We can find a basis of $\mathfrak n$ such that each $e_i$ belongs to a single $\mathfrak g_\alpha$.
The fact that $A$ acts diagonally on $\mathfrak n$ corresponds to the diagonal action of the (positive simple) roots: $[a,e_i] = \sum \alpha_{j}(a)d_{i,j} e_i$.

\paragraph{Example 3}
The similarity geometries defined in~\cite{Ale} are the ray geometries of rank one with the property that for every direction $e_i$, $d_{i}>0$. It is the case for any Carnot group.

\paragraph{Notations}
The action of $G$ on $N$ is given  by
\begin{equation}
(c,f)(x) = L_c(\exp(f(\ln(x)))).
\end{equation}
In order to simplify the notations, if $x\in N$ and $f\in KA$, we define
\begin{equation}
f(x)\coloneqq \exp(f(\ln(x))).
\end{equation}
We denote by $+$ the group law of $N$. We should be careful that $+$ is not commutative in general, in opposition with the addition of $\mathfrak n$. If $f$ is an automorphism of $\mathfrak n$, then
\begin{equation}
\forall x,y\in N, \; f(x+y)=f(x)+f(y).
\end{equation}
(Note that this is an example of the use of $+$ for the group law of $N$.)

For any $x\in N$ and any $t\in \R$, we define
\begin{equation}
tx = \exp(t\ln(x)).
\end{equation}
 Note that if $f$ is an automorphism of $\mathfrak n$, then  by linearity of $f$ on $\mathfrak n$,
 \begin{equation}
\forall x,y\in N, \forall t\in \R, \; f(x+ty) = f(x)+tf(y).
 \end{equation}

Therefore, for any nil-affine transformation $T\in N\rtimes  \GL_{\Aut}(\mathfrak n)$,
 we can express $T$ as 
\begin{equation}
T(x) = c + f(x)
\end{equation}
where $c\in N$ and $f\in \GL_{\Aut}(\mathfrak n)$. We can see that geodesics are preserved:
\begin{equation}
T(x + tv) = c + f(x+tv) = c + f(x) + f(tv) =( c+f(x) )+ tf(v).
\end{equation}
Note also that if we change the base point from $0=e\in N$ to any $y\in N$ then
\begin{equation}
T(x) = T(y-y+x)  = c + f(y) + f(-y+x),
\end{equation}
so $f$ is constant but $c$ is changed to $c+f(y)$. It characterizes the semi-direct product.

\section{Ray nil-affine geometries in flag manifolds}\label{sec-3}

Consider $G$ a non-compact semi-simple Lie group with finite center.
Say that $P$ is a parabolic subgroup. Up to conjugation, it corresponds to a choice $\Sigma\subset \Delta$ of simple roots and to a grading (see~\cite[p. 309]{Cap})
\begin{equation}
\mathfrak g = \mathfrak \mathfrak g_{-k,\Sigma}\oplus \dots \oplus \mathfrak g_{-1,\Sigma}\oplus \mathfrak g_{0,\Sigma}\oplus \mathfrak g_{1,\Sigma}\oplus\dots\oplus \mathfrak g_{k,\Sigma}.
\end{equation}
The parabolic subgroup $P$ corresponds to the Lie algebra of the null positive grades:
\begin{equation}
\mathfrak p = \mathfrak g_{0,\Sigma}\oplus \mathfrak g_{1,\Sigma}\oplus\dots\oplus \mathfrak g_{k,\Sigma}
\end{equation}
and an opposite parabolic subgroup $Q$ is given by the null and negative grades:
\begin{equation}
\mathfrak q = \mathfrak g_{0,\Sigma}\oplus \mathfrak g_{-1,\Sigma}\oplus\dots\oplus \mathfrak g_{-k,\Sigma}
\end{equation}

The positive and negative grades give two isomorphic nilpotent subalgebras 
\begin{align}
\mathfrak n_{-\Sigma} &= \mathfrak g_{-1,\Sigma}\oplus\dots\oplus \mathfrak g_{-k,\Sigma},\\
\mathfrak n_{\Sigma} &= \mathfrak g_{1,\Sigma}\oplus\dots\oplus \mathfrak g_{k,\Sigma}.
\end{align}
The subalgebra
\begin{equation}
\mathfrak l = \mathfrak g_{0,\Sigma} = \mathfrak p \cap  \mathfrak q
\end{equation}
is the Levi factor and corresponds to the subgroup $L = P\cap Q$.

One can take the (opposite) nilpotent subgroups $N_+\subset P$ and $N_-\subset Q$ corresponding to the subalgebras $\mathfrak n_\Sigma$ and $\mathfrak n_{-\Sigma}$ respectively. There is a natural embedding $N_+\to G/Q$ (see~\cite[p. 20-21]{Guivarch}) and the image is known as an \emph{open Schubert stratum}~\cite[p. 175]{Kapovich}.  This open subset is dense inside $G/Q$.

So one can consider the subgeometry corresponding to $N_+\subset G/Q$. The isotropy is $L$. The subgroup $L$ acts on $N_+$ by conjugation since $Q = LN_-$ and $N_+$ acts on itself by left-translation.

This can be summarized by the following proposition.
\begin{proposition}
Let $G$ be a non-compact semi-simple Lie group with finite center and let $P\subset G$ be a parabolic subgroup.  Let $Q\subset P$ be an opposite parabolic subgroup. Decompose $P = LN_+$ and $Q = LN_-$. Then $(P,N_+)$ is a subgeometry of $(G,G/Q)\cong (G,G/P)$ and $N_+$ is an open and dense subset of $G/Q$.
\end{proposition}

\begin{definition}
We call $(P,N_+)$ a \emph{Levi subgeometry}.
\end{definition}

In general, a Levi subgeometry $(P,N_+)$ is not a ray nil-affine geometry, except in the case where $P=B$ is the Borel subgroup. (Note that $P=B$ corresponds to $\Sigma=\emptyset$.) With Langland's decomposition
\begin{equation}
B = M_BA_BN_+, \; \mathfrak b = \mathfrak m_B\oplus\mathfrak a_B\oplus \mathfrak n_B
\end{equation}
it is clear that $M_BA_B$ verifies the conditions to be a ray nil-affine isotropy group. (Note that we always have $\mathfrak n_\Sigma\subset\mathfrak n_B$.) In general, $\mathfrak p$ is decomposed as (see~\cite{Guivarch, Knapp}):
\begin{equation}
P = M(\Sigma)A(\Sigma)N_+, \; \mathfrak p = \mathfrak m_\Sigma\oplus\mathfrak a_\Sigma\oplus \mathfrak n_\Sigma
\end{equation}
and $M(\Sigma)A(\Sigma)$ always contains $M_BA_B$. The subgroup $A(\Sigma)$ still acts by diagonal transformations but $M(\Sigma)$ (that centralizes $A(\Sigma)$ in $L$) is no longer compact. For instance, $\mathfrak m_\Sigma\cap \mathfrak a_B$ has positive dimension if $P\neq B$. More specifically:
\begin{equation}
\mathfrak a_\Sigma = \{H\in\mathfrak a\; | \forall \phi\in \Sigma,\;\phi(H)=0\}.
\end{equation}

We can denote $K\subset G$ a maximal compact such that $M_B\subset K$. If we consider $K\cap M_\Sigma$ then it still centralizes $A_\Sigma$ and becomes compact.

The construction we  apply gives two different ray nil-affine geometries. It is summarized by the following.
\begin{proposition}\label{prop-pararay}
Let $(P,N_+)$ be a Levi subgeometry. We have two ray nil-affine geometries (that are different if $P\neq B$):
\begin{align}
\left(N_+\rtimes M_BA_B,  N_+\right) ,\\
\left(N_+ \rtimes (K\cap M_\Sigma)A_\Sigma, N_+\right).
\end{align}
The rank of the first ray geometry is the real rank of the underlying symmetric space. The rank of the second is the rank of the first minus the cardinal of $\Sigma$.
\end{proposition}

\section*{Example: real non-compact forms of $\SL(3,\C)$ and $\SL(4,\C)$}

As an example, we give a table of the relevant characteristics of the parabolic geometries of the non-compact real forms of $\SL(3,\C)$ and $\SL(4,\C)$. It shows that the  Levi subgeometries of dimension less than six have a nilpotency order of two, and therefore our results can be applied. Note however that in higher nilpotency order we still have a similar classification result on the Carnot group~\cite{Ale}.

\begin{table}[htp]
\begin{center}
\begin{tabular}{|c|c|c|c|c|}
\hline
$G$ and real rank& $\Sigma\subset \Delta$ &  $\dim(\mathfrak n_\Sigma)$ & Nil-order \\
\hline
\hline
$\SL(3,\R)$, rank $2$ & $\emptyset$    & $3$ &  $2$ \\
& $\{\phi_i\}$  & $2$ & $1$
\\ \hline 
$\SU(2,1)$, rank $1$ &$\emptyset$    & $3$ & $2$  
\\ \hline
\hline
$\SL(4,\R)$, rank $3$ & $\emptyset$   & $6$ & $3$ \\
& $\{\phi_i\}$ & $5$ & $2$ \\
& $\{\phi_i,\phi_{i+1}\}$ & $4$ & $2$ 
\\ \hline
$\SU(3,1)$, rank $1$ & $\emptyset$  & $5$ & $2$
\\ \hline
$\SU^*(4)$, rank $1$ & $\emptyset$  & $4$ & $1$
\\ \hline
$\SU(2,2)$, rank $2$ & $\emptyset$  & $6$ & $3$ \\
 & $\{\phi_2\}$  & $5$ & $2$ \\
 & $\{\phi_1\}$  & $4$ & $1$
\\ \hline
\end{tabular}
\end{center}
\caption{The parabolic geometries of the real forms of $\SL_3(\C)$ and $\SL_4(\C)$.}
\end{table}%

\subsection*{The case of $\SU(2,2)$}
We workout a full example: the group $\SU(2,2)$.
We choose the  Hermitian form with signature $(2,2)$ determined by
\begin{equation}
J=\begin{pmatrix}0 & 0 & 0 & 1 \\ 0 & 0 & 1 & 0 \\ 0 & 1 & 0  & 0 \\ 1 & 0 & 0 & 0\end{pmatrix}.
\end{equation}
We let, $x,y,z,w\in \C$ and $\alpha_1,\alpha_2,\beta,\gamma,\Delta,s,t\in \R$. The Lie algebra of $\SU(2,2)=\SU(J)$ is described as follows.

\begin{equation}
\mathfrak{su}(2,2) = \begin{pmatrix}
\alpha_1 - \ii \beta & -\overline x & -\overline y & \ii t \\
-\overline z & \alpha_2 + \ii \beta & \ii \gamma & y \\
-\overline w & \ii \delta & -\alpha_2 + \ii \beta & x \\
\ii s & w & z & -\alpha_1 - \ii\beta
\end{pmatrix}
\end{equation}

By choosing the roots
\begin{equation}
\phi_1 = \alpha_1-\alpha_2 , \; \phi_2 = 2\alpha_2
\end{equation}
we can describe the structure by replacing variables with their corresponding root space.

\begin{equation}
\begin{pmatrix}
\mathfrak h & \phi_1 & \phi_1 + \phi_2 & 2\phi_1 + \phi_2 \\
-\phi_1 & \mathfrak h & \phi_2 & \phi_1+\phi_2 \\
-(\phi_1+\phi_2) & -\phi_2 & \mathfrak h & \phi_1 \\
-(2\phi_1+\phi_2) & -(\phi_1+\phi_2) & -\phi_1 & \mathfrak h
\end{pmatrix}
\end{equation}

So, for example, the coordinate $\ii t$ in the first matrix corresponds to the root space associated to the root $2\phi_1+\phi_2$ as the second matrix indicates. 

\vskip10pt
Now, they are several choices of parabolic geometries.
Each gives a different dimension for $\mathfrak n_\Sigma$.

\paragraph{The Borel parabolic subgroup}The case $\Sigma=\emptyset$ corresponds to the choice of the Borel subgroup.
\begin{equation}
\mathfrak b = 
\underbrace{
\begin{pmatrix}
 - \ii \beta & && \\
 &  + \ii \beta && \\
 &  &  + \ii \beta &  \\
&  &  &  - \ii\beta
\end{pmatrix}
}_{\mathfrak m_B} 
\oplus
\underbrace{
\begin{pmatrix}
\alpha_1 & && \\
 & \alpha_2  && \\
 &  & -\alpha_2  &  \\
&  &  & -\alpha_1 
\end{pmatrix}
}_{\mathfrak a_B}
\oplus
\underbrace{
\begin{pmatrix}
0 & -\overline x & -\overline y & \ii t \\
 &0 & \ii \gamma & y \\
 &  & 0 & x \\
&  &  &0
\end{pmatrix}
}_{\mathfrak n_B}
\end{equation}
In the coordinates $(x,\gamma,y,t)$ of $\mathfrak n_B$ we can express the adjoint action of $\mathfrak h_B=\mathfrak m_B\oplus\mathfrak a_B$ by:
\begin{equation}\label{eq-hbn}
\mathfrak h_B|_{\mathfrak n_B}:
\begin{pmatrix}
(\alpha_1-\alpha_2) +2\ii\beta \\
& 2  \alpha_2 \\ 
&& (\alpha_1+\alpha_2) +2\ii\beta \\
&&&2\alpha_1
\end{pmatrix}
\end{equation}
and it shows of course that it is indeed a ray geometry.

\paragraph{The $5$-dimensional parabolic geometry}
We obtain the $5$-dimension parabolic geometry by setting $\Sigma=\{\phi_2\}$.
It corresponds to a CR geometry (that is in fact contact).
This parabolic geometry is also the boundary geometry related to an  homogeneous space in $\CP^3$. Indeed, consider the equation
\begin{equation}
\Re(x_1\overline{x_4}) + \Re(x_2\overline{x_3}) < 0.
\end{equation}
It describes an homogeneous space with semi-simple group $\SU(2,2)$. The boundary
\begin{equation}
\Re(x_1\overline{x_4}) + \Re(x_2\overline{x_3}) = 0
\end{equation}
is a real submanifold of $\CP^3$ having dimension $5$. The parabolic subgroup that we choose corresponds to the stabilizer of $(1:0:0:0)$. (Compare this geometry with~\cite{Kassel}.)

The Lie algebra is given by:
\begin{equation}
\mathfrak{p}_{\phi_2} = 
\underbrace{
\begin{pmatrix}
\alpha_1-\ii\beta \\ 
& \alpha_2+\ii\beta & \ii\gamma \\
&\ii\delta & -\alpha_2+\ii\beta \\
&&&-\alpha_1-\ii\beta
\end{pmatrix}
}_{\mathfrak l_{\phi_2}}
\oplus
\underbrace{
\begin{pmatrix}
 0& -\overline x & -\overline y & \ii t \\
 &0&0&y\\
 &&0& x \\
 &&&0
\end{pmatrix}
}_{\mathfrak n_{\phi_2}}
\end{equation}

Now, to obtain ray nil-affine geometries, we need to reduce the Levi factor $\mathfrak l_{\phi_2}$ into appropriated combinations of $\mathfrak m \oplus\mathfrak a$.

The first choice consists in taking the same diagonal part as the Borel subgroup, namely $\mathfrak h_B\subset\mathfrak l_{\phi_2}$.
The adjoint action of $\mathfrak h_B$ is described by equation~\eqref{eq-hbn}.

The second choice consists in setting $\phi_2=2\alpha_2=0$ and in
choosing a compact factor in the middle square. We obtain the following Lie algebra.
\begin{equation}
\mathfrak h_{\phi_2} =
\underbrace{
\begin{pmatrix}
-\ii\beta \\
&\ii\beta & \ii\gamma\\
&\ii\gamma & \ii\beta \\
&&&-\ii\beta
\end{pmatrix}
}_{\mathfrak m_{\phi_2}}
\oplus
\underbrace{
\begin{pmatrix}
\alpha_1 \\
&0\\
&&0\\
&&&-\alpha_1
\end{pmatrix}
}_{\mathfrak a_{\phi_2}}\\
\end{equation}
And the adjoint action, with $(x,y,t)$ as basis of $\mathfrak n_{\phi_2}$, is described by:
\begin{equation}
\mathfrak h_{\phi_2}|_{\mathfrak n_{\phi_2}} : 
\begin{pmatrix}
\alpha_1 + 2\ii\beta &  \ii \gamma \\
\ii\gamma & \alpha_1+2\ii\beta \\
&&2\alpha_1
\end{pmatrix}
\end{equation}

\paragraph{The $4$-dimensional parabolic geometry}
The last parabolic geometry is given by $\Sigma=\{\phi_1\}$. We obtain the following Lie algebra.
\begin{equation}
\mathfrak p_{\phi_1} = 
\underbrace{
\begin{pmatrix}
\alpha_1-\ii\beta & -\overline x\\ 
-\overline z& \alpha_2+\ii\beta & \\
& & -\alpha_2+\ii\beta &x \\
&&z&-\alpha_1-\ii\beta
\end{pmatrix}
}_{\mathfrak l_{\phi_1}}\oplus
\underbrace{
\begin{pmatrix}
0 & 0 & -\overline y & \ii t \\
 &0 & \ii \gamma & y \\
 &  & 0 & 0 \\
&  &  &0
\end{pmatrix}
}_{\mathfrak n_{\phi_1}}
\end{equation}

The first ray geometry is again obtained by taking the Borel Levi factor. The adjoint action of $\mathfrak h_B\subset\mathfrak l_{\phi_1}$ on $\mathfrak n_{\phi_1}$ is described by equation \eqref{eq-hbn}.

The second choice consists in reducing $\mathfrak a$ to the subspace $\mathfrak a_\Sigma$ where $\phi_1=0$, that is to say, with $\alpha_1=\alpha_2$ since $\phi_1=\alpha_1-\alpha_2$. 
It gives the following algebra
\begin{equation}
\mathfrak h_{\phi_1} = 
\underbrace{
\begin{pmatrix}
-\ii \beta & -\overline x \\
x & \ii\beta \\
&& \ii\beta  & x\\
&&-\overline x&-\ii\beta
\end{pmatrix}}_{\mathfrak m_{\phi_1}}
\oplus
\underbrace{
\begin{pmatrix}
\alpha_1 \\
&\alpha_1 \\
&&-\alpha_1 \\
&&&-\alpha_1
\end{pmatrix}}_{\mathfrak a_{\phi_1}}.
\end{equation}
The adjoint action of $\mathfrak h_{\phi_1}$ is more difficult to express as a single matrix. On $\mathfrak h_{\phi_1}\cap \mathfrak h_B$, we already know the result by equation~\eqref{eq-hbn}. So it lasts the $\mathfrak{su}(2)$ part: 
\begin{equation}
\left[
\begin{pmatrix} 
0 & -\overline{x} \\
x & 0 \\
&&0&x \\
&&-\overline x & 0
\end{pmatrix},
\begin{pmatrix}
0 & 0 & -\overline y & \ii t \\
 &0 & \ii \gamma & y \\
 &  & 0 & 0 \\
&  &  &0
\end{pmatrix}
\right] = 
\begin{pmatrix}
0 & 0 & \ii(t - \gamma)\overline x & 2\ii{\rm Im}(x\overline y)  \\
 &0 & 2 \ii{\rm Im}(y\overline{x})& \ii(t - \gamma)x\\
 &  & 0 & 0 \\
&  &  &0
\end{pmatrix}.
\end{equation}

\section{Ray closed manifolds}\label{sec-4}

In this section we prove the following.
\begin{theorem}\label{thm-partialcomp}
\thmpartialcomp
\end{theorem}

And in fact, the proof is very similar to the proof in~\cite{Ale3}. We will not repeat the proofs  in~\cite{Ale3} that can directly be  used in our setting. We will recall the construction and the notations.

\subsection{Fried dynamics}

Let $M$ be an incomplete closed nil-affine manifold. Choose any metric on $M$ compatible with its topology.
Denote
 $\pi\colon \widetilde M \to M$ its universal cover.
If $x\in\widetilde M$ and $\gamma\subset \widetilde M$ is a geodesic issued from $x$, incomplete at $t=1$, then the projection of $\gamma$ in $M$ is an open curve without any continuous completion at $t=1$.
Since $M$ is closed, the projection $\pi(\gamma)$ has a recurrent point $y\in M$.

Let $U\subset M$ be a compact neighborhood of $y$, convex and trivializing the universal cover.
For the choice of a decreasing sequence $\epsilon_i \to 0$ we can define $t_i$ the time such that $\pi(\gamma(t_i))$ belongs to $U$ and is at distance at most $\epsilon_i$  to $y\in M$. We ask that $\pi(\gamma)$ exits $U$ between the times $t_i$ and $t_{i+1}$. We have that $t_i\to 1$ since the geodesic is incomplete at $t=1$.

We use the trivialization of the universal cover by $U$. It provides $U_i\subset \widetilde M$ such that $\gamma(t_i)\in U_i$. Let $y_i\in U_i\subset \widetilde M$ be the lifts of $y\in U\subset M$.
Every $U_i$ is again a compact convex neighborhood of $y_i$.
Through the developing map $D\colon \widetilde M \to \cN$, each $D(U_i)$ is compact, convex and they accumulate along the compactification $\overline{D(\gamma)}$ of $D(\gamma)$ at  $t=1$.
Since the intersections of $\gamma$ with the $U_i$'s are transverse and disjoint along $\gamma$, the $D(U_i)$'s intersect transversally and disjointly $D(\gamma)$.

\begin{definition}
Let $\gamma\subset \widetilde M$ be an incomplete geodesic at $t=1$. Let $y\in M$ be a recurrent point of its projection into $M$. Let $U$ be a convex, compact, trivializing neighborhood of $y$. Let $\epsilon_i\to 0$ be decreasing and let
$t_i\to 1$ be an associated sequence of times such that  (for the lifts $y_i\in U_i$ of $y\in U$) we have $\gamma(t_i)\in U_i$ and the distance between $\pi(\gamma(t_i))$ and $y$ is lesser than $\epsilon_i$.
Define $g_{ji}\in \pi_1(M,y)$ the transformations of $\widetilde M$ verifiying
\begin{equation}
g_{ji}(U_i) = U_j.
\end{equation}
Those data define a \emph{Fried dynamics}.
\end{definition}

Note that a subsequence of the  times $\{t_i\}$ (or the distances $\{\epsilon_i\}$) corresponds univocally to a subsequence of the pairs  $\{y_i\in U_i\}$.

The transformations $g_{ji}$ verify a cocycle property: 
\begin{equation}
g_{ki} = g_{kj}g_{ji}.
\end{equation}
For the nil-affine geometry $(G,\cN)$ considered, denote by $T_{ji}$ the corresponding transformations by the holonomy morphism $T_{ji}=\rho(g_{ji})\in G$.

\vskip10pt
We  assume that $(G,\cN)$ is a ray geometry.
We have a basis $(e_1,\dots,e_n)$ of $\mathfrak n$ such that $A\subset KA\subset G$ acts diagonally. 
Once a base point of $\cN$ is chosen, we can express any $T_{ji}\in G$ by 
\begin{equation}
T_{ji}(x) = c_{ji} + f_{ji}(x),
\end{equation}
with $c_{ji}\in \cN$ and $f_{ji}\in KA$.
Recall that a change of the base point  is expressed by
\begin{equation}
 T_{ji}(x)=T_{ji}(y-y+x)  = (c_{ji} + f_{ji}(y)) + f_{ji}(-y+x).
\end{equation}
It preserves the linear part $f_{ji}$.
Decompose each $f_{ji}$ into 
\begin{equation}
f_{ji} = f_{ji,K}f_{ji,A}
\end{equation}
with $f_{ji,K}\in K$ et $f_{ji,A}\in A$. 
Recall that $K$ centralizes $A$, so both factors commute. The cocycle property on $T_{ji}$ implies that each factor also verifies the cocycle relation:
\begin{equation}
f_{ki,K} = f_{kj,K}f_{ji,K}, \; f_{ki,A} = f_{kj,A}f_{ji,A}
\end{equation}

\begin{lemma}
Up to  a subsequence of $\{y_i\in U_i\}$, we have
\begin{equation}
\lim_{j\to \infty}\lim_{i\to \infty} f_{ji,K} = \id.
\end{equation}
Let $e_q$ be a basis vector. Denote by $\beta_{ji,q}$ the diagonal element of $f_{ji,A}\in A$ for the direction $e_q$. Up to  a subsequence of $\{y_i\in U_i\}$, we have
\begin{equation}
\lim_{j\to \infty}\lim_{i\to \infty}  \beta_{ji,q} = \omega_q\in \{0,1,\infty\}.
\end{equation}
\end{lemma}

It should be noted that if
$\beta_{kj,q}\to \omega_q$ when $k\gg j\to \infty$, then we can obtain an information on how $\beta_{ki,q}$ can evolve when $i$ is fixed and $k\to \infty$. Indeed, we have
\begin{equation}
\lim _{k\to \infty}\beta_{ki,q} = \lim_{k\to\infty }\lim_{j\to\infty} \beta_{ki,q} = \lim_{k\to \infty}\lim_{j\to\infty} \beta_{kj,q}\beta_{ji,q}= \omega_q \lim_{j\to \infty}\beta_{ji,q}.
\end{equation}
If $\beta_{ki,q}\to r$ when $k \to \infty$ (up to choose a subsequence) then $r=\omega_qr$. If $\omega_q\in\{0,\infty\}$ it implies $r=\omega_q$.

\begin{lemma}\label{lem-valprop}\label{lem-asympt}
Let $i>0$ fixed.  If $\beta_{ki,q}\to r$  and $\omega_q=1$ then $r$ must be a (finite) real positive number.
\end{lemma}

\begin{proposition}
There exists a subsequence of $\{y_i\in U_i\}$ such that:
\begin{itemize}
\item for $i>0$ fixed, $f_{ji,K}$ converges;
\item for  $i>0$ fixed, the sequence $\{\beta_{ji,q}\}$ is monotonic or constant for $j>i$.
\end{itemize}
\end{proposition}

\begin{definition}
We define a linear decomposition $E\oplus P \oplus F$ of $\mathfrak n$ by deciding that $e_q\in E$ if $\omega_q=0$, $e_q\in P$ if $\omega_q=1$ and $e_q\in F$ if $\omega_q=\infty$.
\end{definition}

Since $A$ acts diagonally and $K$ centralizes $A$, we have that $f_{ji}$ preserve the decomposition $E\oplus P \oplus F$. Note that $\dim E>0$ since $D(U_i)$ accumulates disjointly on $\overline{D(\gamma)}$. The other subspaces might be reduced to $\{0\}$.

If we chose any base point $p$, it gives three nil-affine subspaces based at $p$:
\begin{equation}
\forall L\in \{E,F,P\} \forall p\in N, \; L|_p \coloneqq L_p\exp(L).
\end{equation}
Note also that for example $[F,F]\subset F$ and $[F\oplus P,F\oplus P]\subset F\oplus P$ since $A$ acts by automorphisms. It shows for instance that $\exp(F)$ is a subgroup of $N$,

\vskip10pt
We come back to $M$ and discuss the choice of $U\subset M$.
Choose $V\subset U$ a smaller neighborhood of $y$. Since we chose distances $\epsilon_i\to 0$, for $i\geq i_0$ large enough, every $\pi(\gamma(t_i))$ will also belong to $V$. Note that we can lift $V\subset M$ into $V_i\subset U_i\subset \widetilde M$.

\begin{lemma}\label{lem-geoy}
For any $i>0$, $g_{ji}^{-1}\gamma$ has $y_i$ for limit point when $j\to \infty$.
\end{lemma}

\begin{definition}
In $M$, let $U_1 = U$. Define for a sequence of $\{0<r<1\}$ tending to $0$ a sequence $\{U_r\}$ of compact convex neighborhoods  of $y\in M$. Assume that $U_r$ is decreasing for the inclusion and that $U_r\to \{y\}$ when $r\to 0$. In $\widetilde M$, define $U_{1,i}=U_i$ and  $U_{r,i}$ the lift of $U_r$ such that $U_{r,i}\subset U_{1,i}$. In $D(\widetilde M)$, define
\begin{equation}
C_{r,i}  = D(U_{r,i}).
\end{equation}
\end{definition}

Recall that $U_1=U$ is a convex, compact  and trivializing neighborhood of $y\in M$.
Every $C_{r,i}$ is convex and compact in $\cN$ since $U_{r}$ is a convex, compact and trivializing neighborhood of $y$ in $M$. 

\begin{lemma}
For any $i,j$ and $r\geq 0$, $g_{ji}U_{r,i} = U_{r,j}$ and by consequence $T_{ji}C_{r,i}=C_{r,j}$.
\end{lemma}

As discussed before, since $U_r\subset U$ is a neighborhood of $y$, $D(\gamma)(1)$ is a limit point of $\{C_{r,i}\}$ since for $j>j_0$  large enough $D(\gamma(t_j))$ belongs to $C_{r,j}$.

\begin{proposition}
Let $r>0$. Any accumulation point of $\{C_{r,i}\}$ is a limit point.
\end{proposition}

\begin{definition}
For any fixed value $r> 0$, let $C_{r,\infty}$ be the limit set  of the sequence $\{C_{r,i}\}$.
\end{definition}

Note that when $r=0$, each $C_{r,i}$ is reduced to $D(y_i)$. Those have no reason to accumulate to $D(\gamma)(1)$.  This is why we asked $r$ to be different from $0$.

\begin{definition}
Define 
\begin{equation}
C_{0,\infty} = \bigcap_{r>0} C_{r,\infty}.
\end{equation}
\end{definition}

This definition makes sense because the $C_{r,\infty}$ are decreasing for the inclusion when $r\to 0$.

\begin{lemma}
All the sets $C_{r,\infty}$ and $C_{0,\infty}$ are convex.
The set $C_{0,\infty}$ is nonempty (it contains $D(\gamma)(1)$) and is a nil-affine subspace.
\end{lemma}

\begin{lemma}
Choose $D(\gamma)(1)$ as base point of $\R^n$.
\begin{align}
\forall r>0, \; C_{r,\infty}&\subset \left(P\oplus F\right)|_{D(\gamma)(1)}\\
C_{0,\infty}&= F|_{D(\gamma)(1)}.
\end{align}
\end{lemma}

\paragraph{Fixed point}Now we return to the general study. The next step is to find an asymptotic fixed point for $T_{ji}$.

Since $E\oplus P$ and $F$  are generated by two subfamilies of $(e_1,\dots,e_n)$, we can apply
lemma~\ref{lem-decomp}. We have that for any $x\in N$, there exists a unique decomposition $x = x_L + x_F$ with $x_L\in \exp(E\oplus P)$ and $x_F\in \exp(F)$.

\begin{lemma}\label{lem-convQ}
Choose $D(\gamma)(1)$ as base point and decompose $T_{ji}(x)=c_{ji} + f_{ji}(x)$ with $f_{ji}\in KA$ and $c_{ji}\in N$. 
Let $c_{ji} = c_{ji,L}+c_{ji,F}$ be the decomposition given by the preceding lemma.
Define
 $Q_{ji}(x)= c_{ji,F} + f_{ji}(x)$.  Then
\begin{equation}
\lim_{j\to \infty} T_{ji}(x)-Q_{ji}(x) = \lim_{j\to\infty} c_{ji,L} =0
\end{equation}
 and therefore this convergence is uniform for $x\in N$. 
\end{lemma}

\begin{proof}
With $L = E\oplus P$ and the decomposition
 $L\oplus F = \mathfrak n$,
we can show this result by showing that the coordinates of $c_{ji,L}$ in the Lie algebra  tend to $0$.
By definition, $c_{ji,L}$ has vanishing coordinates on $F$.

For any $\epsilon >0$, 
$T_{ji}(C_{\epsilon,i})$ has for limit $C_{\epsilon,\infty}$ when $j\to \infty$. 
When $\epsilon\to 0$ and $j\to \infty$, the set of the coordinates of
 $T_{ji}(C_{\epsilon,i})$ all tend to $0$ in the linear subspace $L\subset \mathfrak n$. Indeed, $C_{0,\infty}=F|_{D(\gamma)(1)}$.

Note that $c_{ji,L}$ tends or not to $0$ when $j\to \infty$ independently from the choice of $\epsilon$.

Now decompose $T_{ji}(C_{\epsilon,i})$, and note that $f_{ji}$ preserves the decomposition $L\oplus F$.
\begin{align}
c_{ji} + f_{ji}(C_{\epsilon,i}) &= c_{ji,L}+c_{ji,F} + f_{ji}(C_{\epsilon,i})_L + f_{ji}(C_{\epsilon})_F \\
&= \left(c_{ji,L} + \exp(\ad_{c_{ji,F}})(f_{ji}(C_{\epsilon,i})_L)\right) + \left(c_{ji,F} + f_{ji}(C_{\epsilon,i})_F\right)
\end{align}
The last term is exclusively in $\exp(F)$ since $[F,F]\subset F$ shows that $F$ is a subgroup.
So the coordinate in $L$ is uniquely determined by the projection in $L$ of 
\begin{equation}
c_{ji,L} + \exp(\ad_{c_{ji,F}})(f_{ji}(C_{\epsilon,i})_L) = c_{ji,L} +\exp\left(f_{ji}(C_{\epsilon,i})_L + [c_{ji,F},f_{ji}(C_{\epsilon,i})]_L\right)
\end{equation}
and the projection in $L$ must tend to $0$ when $j\to \infty$ and $\epsilon\to 0$.

For any $j>i$, $f_{ji}$ acts as a contraction on $(C_{\epsilon,i})_L$. 
Note that in two-step nilpotency, if the center does not intersect $L$ then $[x,y]_L=0$. By consequence, in either case, we can let $\epsilon v\in (C_{\epsilon,i})_L$ (in the center or not) so that $f_{ji}(\epsilon v) + [c_{ji,F},f_{ji}(\epsilon v)]_L = f_{ji}(\epsilon v)$ and it can only tend to $0$ when $j\to \infty$ and $\epsilon \to 0$.
By consequence, when $j\to \infty$ and $\epsilon\to 0$, $c_{ji,L}\to 0$ since the sum tends to zero. It proves the lemma.
\end{proof}

Note that $Q_{ji}$ has a fixed point on $F$, since $F$ is preserved and $f_{ji}$ acts by expansions on it.

\begin{lemma}
Denote $q_{ji}\in F$ the fixed point of $Q_{ji}$.
 For $i>0$, $q_{ji}$ converges when $j\to \infty$ and we denote  $q_i=\lim q_{ji}$. Also, $D(y_i)\in E|_{q_i}$.
\end{lemma}

\paragraph{Asymptotic dynamic}
Now we can determine the limits of the orbits in positive time, and apply proposition~\ref{prop-visconv}.

\begin{lemma}
Let $z\in E|_{q_i}$ and $V$ a neighborhood  of $z$. Then $\{T_{ji}(V)\}$ has $F|_{q_i}$ in its limit set.
\end{lemma}

\begin{proposition}\label{prop-eplus}
Let $S\subset \widetilde M$ be a convex containing the incomplete geodesic $\gamma$. Assume that $D(S)$ has a smooth boundary at $D(\gamma)(1)$. Let $i>0$. We have the following properties. 
\begin{itemize}
\item
The orbit $T_{ji}^{-1}(D(S))$ tends  to a product $E_{+,i}\times P_c$ of a half-space $E_{+,i}\subset E|_{q_i}$ and a neighborhood $P_c\subset P|_{q_i}$ of the origin.
\item
The product $E_{+,i}\times P_c$ is visible from $y_i$.
\item
The boundary of $E_{+,i}$ is  described by the limit of
 $T_{ji}^{-1}(\rT_{D(\gamma)(1)}D(S)) \cap E|_{q_i}$.
\end{itemize}
\end{proposition}

To get a better description of $E_{+,i}$ inside $E|_{q_i}$, we must explain how to approximate $T_{ji}^{-1}$ relatively to $Q_{ji}^{-1}$. Note that \emph{a priori}
\begin{equation}
T_{ji}^{-1}(x) -Q_{ji}^{-1} (x)= - f_{ji}^{-1}(c_{ji,L}) 
\end{equation}
and might not be tending to zero, or even stay bounded.

The study of $-f_{ji}^{-1}(c_{ji,L})$ is technical. With the two following lemmas, we   will prove  that $-f_{ji}^{-1}(c_{ji,L})\to 0$ when $j\gg i\to \infty$.

\begin{lemma}
Assume that $E_{+,i}$ does not contain $q_i$. 
For any $i>0$, there exists $M>0$ such that for any $j>i$, 
\begin{equation}
c_{ji,L} = f_{ji}(b_{ji,E}) + b_{ji,P}
\end{equation}
 with $b_{ji,E}\in E$, $b_{ji,P}\in P$ verifying $b_{ji,P}\to 0$ and $\|b_{ji,E}\|<M$.
\end{lemma}

We will prove in the case of a rank one ray geometry that $E_{+,i}$ never contains $q_i$. So in fact $b_{ji,E}$ is bounded.
The next lemma explains that in fact $T_{ji}^{-1}$ is as well approximated by $Q_{ji}^{-1}$ as $i>0$ gets large.

\begin{lemma}\label{lem-approxneg}
Assume that for any $i$, $E_{+,i}$  never contains $q_i$. Then 
\begin{equation}
\lim_{k\to \infty}\lim_{j\to \infty}b_{kj,E} = 0.
\end{equation}
\end{lemma}
\begin{proof}
The cocycle relation $T_{kj}T_{ji}=T_{ki}$ gives:
\begin{align}
T_{ji} &=  f_{ji}(b_{ji,E}) + b_{ji,P} + c_{ji,F}+f_{ji} \\
T_{kj}T_{ji} &= f_{kj}(b_{kj,E})  + b_{kj,P}  +c_{kj,F} + f_{kj}(f_{ji}(b_{ji,E}))  + f_{kj}(b_{ji,P})  + f_{kj}(c_{ji,F}) + f_{kj}f_{ji} \\
&= f_{kj}(b_{kj,E})  + b_{kj,P}  +c_{kj,F} + f_{ki}(b_{ji,E})  + f_{kj}(b_{ji,P})  + f_{kj}(c_{ji,F}) + f_{ki} \\
&= T_{ki} =  f_{ki}(b_{ki,E}) + b_{ki,P} + c_{ki,F} + f_{ki}.
\end{align}
and it implies by identification:
\begin{equation}
f_{ki}(b_{ki,E})  =  f_{kj}(b_{kj,E})  +\left( \exp(\ad_{b_{kj,P}})\exp(\ad_{c_{kj,F}})f_{ki}(b_{ji,E})\right)_E .
\end{equation}
Note that $b_{kj,P}\to 0$ when $k\to \infty$. So we omit this term. Also in two-step nilpotency,
\begin{align}
\left(\exp(\ad_{c_{kj,F}})f_{ki}(b_{ji,E})\right)_E 
&=
\exp\left(f_{ki}(b_{ji,E}) + [c_{kj,F},f_{ki}(b_{ji,E})]_E\right)\\
&= f_{ki}\left(\exp\left(b_{ji,E} +[f_{ki}^{-1}(c_{kj,F}),b_{ji,E}]_E\right)\right).
\end{align}

Hence we obtain:
\begin{align}
b_{ki,E} &= f_{ji}^{-1}(b_{kj,E})  + \exp\left(b_{ji,E} +[f_{ki}^{-1}(c_{kj,F}),b_{ji,E}]_E\right) + o(1) \\
&= \exp\left( f_{ji}^{-1}(b_{kj,E}) +   b_{ji,E}  + \frac 12 [ f_{ji}^{-1}(b_{kj,E}),b_{ji,E}] +[f_{ki}^{-1}(c_{kj,F}),b_{ji,E}]_E \right) + o(1).
\end{align}
We can decompose $E$ into a grading $E_1\oplus E_2$ so that the center intersect $E$ exactly at $E_2$. Note that in two-nilpotency $[x,y]_E \in E_2$.

Now we look at the coordinates of each side inside $E_1$. On the left, we have the coordinates in $E_1$ of $b_{ki,E}$, which must be bounded. On the right, we get the coordinates in $E_1$ of $ f_{ji}^{-1}(b_{kj,E})+ b_{ji,E}$. It can not be bounded unless the coordinates in $E_1$ of $b_{kj,E}$ tend to zero since $f_{ji}^{-1}$ expands.

 It implies also that $[f_{ji}^{-1}(b_{kj,E}),b_{ji,E}]$ is bounded since it only depends of the components in $E_1$.

In $E_2$, we get on the left the coordinates in $E_2$ of $b_{ki,E}$ that is again bounded. On the right, we get the coordinates in $E_2$ of $ f_{ji}^{-1}(b_{kj,E})+ b_{ji,E}$ and the sum (completely in $E_2$) $\frac 12 [ f_{ji}^{-1}(b_{kj,E}),b_{ki,E}] +[f_{ki}^{-1}(c_{kj,F}),b_{ji,E}]_E$. The term $b_{ji,E}$ is bounded. The first bracket is bounded by what precedes. It lasts to prove that the second bracket is bounded.

Note that since $Q_{mn}$ fixes $q_{mn}\in F|_{D(\gamma)(1)}$ we have
\begin{equation}
c_{mn,F} = q_{mn} - f_{mn}(q_{mn}).
\end{equation}
Also, the expression of the cocycle $T_{kj}T_{ji}=T_{ki}$ and the two-step nilpotency show that
\begin{align}
\left[c_{kj,F} + f_{kj}(c_{ji,F}),\cdot \right]& = \left[c_{ki,F},\cdot\right]\\
\left[f_{ki}^{-1}(c_{kj,F}), \cdot\right] &= \left[ f_{ki}^{-1}(c_{ki,F}) - f_{ji}^{-1}(c_{ji,F}), \cdot \right]\\
&= \left[f_{ki}^{-1}(q_{ki}) - q_{ki} - f_{ji}^{-1}(q_{ji}) + q_{ji},\cdot\right]
\end{align}
and when $j\to\infty$ and $k\to\infty$ we have
\begin{equation}
\lim f_{ki}^{-1}(q_{ki}) - q_{ki} - f_{ji}^{-1}(q_{ji}) + q_{ji} = 0 - q_{i} - 0 + q_{i} = 0.
\end{equation}
By consequence, the second bracket tends to zero, and is therefore bounded.

Hence $f_{ji}^{-1}(b_{kj,E}) $
is bounded and it implies that
 $b_{kj,E}\to 0$.
\end{proof}

\subsection*{Rank one ray  manifolds}

We examine closed manifolds with a rank one ray geometry. For those manifolds, the holonomy takes its values in $G_1=N\rtimes KA_1$, where $A_1$ has dimension $1$.

\begin{lemma}
If $(G_1,\cN)$ is a rank one ray geometry, then there exists a decomposition $L_1\oplus L_2\oplus L_3$ such that for any Fried dynamics, $P = L_3$ and $\{E,F\}=\{L_1, L_2\}$.
\end{lemma}

A consequence of this observation is that $P$ is independent from the dynamics, since $d_i=0$ is a condition on $A=A_1$. Therefore, from a Fried dynamics to another one $E$ and $F$ are either the same or exchanged. Also, every direction in $P$ is completely visible by the following lemma.

\begin{lemma}
The direction vector of an incomplete  geodesic $D(\gamma)$ has a non-vanishing coordinate along $E$ in the linear decomposition $\mathfrak n=E\oplus P \oplus F$ associated to its Fried dynamics.
\end{lemma}

\begin{proposition}
Let $S\subset \widetilde M$ be a convex containing $\gamma$. Assume that $D(S)$ has smooth boundary at $D(\gamma)(1)$.
Let $i>0$.
The subspace $H_i = E_{+,i}\times (P\oplus F)$ is visible from $y_i$  (this is a half-space of $\cN$).
\end{proposition}

Here we take the cartesian product in the sense of the  group law. A point of $A\times B$ is a point $a+b$ with $a\in A$ and $b\in B$.

\begin{proof}
We use here the fact that a two-step nilpotent space has the property that any affine subset (in the classical sense) is a nil-affine subset and conversely.

By  proposition~\ref{prop-eplus}, $E_{+,i}$ is a visible  from $y_i$. Since the directions in $P$ are always complete, the product $E_{+,i}\times P$ is fully visible from $y_i$. We show that we can extend $E_{+,i}\times P$ to a visible open $H_i$ containing $E_{+,i}\times (P\oplus F)$.

Let $K\subset E_{+,i}\times P$ be convex and visible from $y_i$. Since the visible space from $y_i$ is open, there exists an open convex $W\subset F$ such that $K\times W $ is visible and convex. 

We show that we can extend $K$ to $E_{+,i}\times P$. 
If that were not the case, there would exists $k+w\in \partial(K) \times W$ such that $z=k+w$ is invisible from $K\times W$. Now, at $z$, $E|_z$ is transverse to $\rT_z (\partial K\times W)$. Therefore, we can consider $u\in E$ such that $z - tu$ is inside $K\times W$ for $t>0$. We let $\eta(t) = z - u + tu$. It is an incomplete geodesic at $t=1$ but lies inside $K\times W$ for $t<1$.

Now, its limit set $C_{0,\infty}$ is contained in $\rT_z(\partial K\times W)$ and therefore can only be $F|_z$.
The limit  set $F|_{z}$ must intersect $E_{+,i}\times P$.
In fact it intersects it at $k$ since, $z= k+w$ and $z+F = k+w+F = k + F$.

Therefore it intersects $E_{+,i}\times P$ and is simultaneously invisible since it is in $C_{0,\infty}'\cap (K\times W)$ and visible since $E_{+,i}\times P$ is completely visible, a contradiction.

Now, $(E_{+,i}\times P)\times W$ can be extended to $H_i$ such that it contains $ E_{+,i}\times (P\oplus F)$. Indeed, apply $T_{ji}$ for $j\to \infty$ (note that $T_{ji}(x+y) = T_{ji}(x)+f_{ji}(y)$):
\begin{align}
\lim_{j\to \infty}  T_{ji}((E_{+,i}\times P)\times W) 
&=(E_{+,i}\times P) \times \lim_{j\to\infty}f_{ji}(W) \\
&= (E_{+,i}\times P) \times F.\qedhere
\end{align}
\end{proof}

\begin{lemma}
The half-spaces $H_i$ tend to a half-space denoted $H_x$ when $i\to \infty$. For $i\to \infty$, $q_i$ gets closer to $\overline{H_i}$. 
\end{lemma}

\begin{lemma}\label{lem-boule}
Let $x\in\widetilde M$. Let $D(S)\subset D(\widetilde M)$ be  the maximal Euclidean open ball such that $x\in S$ and $S$ is visible from $x$. Then $H_x = \lim H_i $ contains $x$.
\end{lemma}

\begin{proof}
We choose the Euclidean norm on $\mathfrak n$ associated to the basis $(e_1,\dots e_n)$. Those Euclidean balls are geodesically convex in two-step nilpotency.

We base $\cN$ at $D(\gamma)(1)$.
To prove $x \in \lim H_i$, we prove that $\langle x,\nu_{D(\gamma)(1)} H_x\rangle  < 0$ with $\nu H_x$ the choice of an exterior normal vector at $D(\gamma)(1)\in\partial H_x$. Note that for $S$ the maximal Euclidean open ball we have $\langle x, \nu_{D(\gamma)(1)} S\rangle <0$ since $x$ is normal to the boundary of $S$ at $D(\gamma)(1)$.

The normal vector at the boundary $\nu_{D(\gamma)(1)}H_x$ is described by $\lim f_{kj}^{-1} (\nu_{D(\gamma)(1)} S)$ when $j\to \infty$ and $k\to \infty$. 
Indeed, $H_x$ is described by $\lim T_{kj}^{-1}(S)$. Let $\Delta$ be the tangent plane to $S$ at $D(\gamma)(1)$. Then $\Delta$ contains $(P\oplus F)|_{D(\gamma)(1)}$. Since $F$ is a subalgebra, and $f_{mn}$ preserves $F$,
\begin{align}
T_{kj}^{-1} (\Delta)  &= Q_{kj}^{-1}(\Delta) + o(1) \\
&= c_{kj,F} + f_{kj}^{-1}(\Delta) + o(1)\\
&= f_{kj}^{-1}(\Delta) + o(1)
\end{align} 

Now, note that when $k\gg j\to \infty$, we get that $f_{kj}$ is essentially its $A$-factor since the rotational part tends to the identity. The map $f_{kj,A}$ acts diagonally with positive factors. 
Denote $f_{kj,A}^{1/2}$ the transformation obtained by taking the square root of each diagonal factor. 

Both $x$ and $\nu_{D(\gamma)(1)}$ are normal to $\Delta\supset (F\oplus P)|_{D(\gamma)(1)}$. So both are contained in $E|_{D(\gamma)(1)}$. 
Note that $f_{kj}^{-1/2}$ acts by arbitrarily strong expansions on $x$ and $\nu_{D(\gamma)(1)}$.
By consequence
\begin{align}
\langle x, \lim f_{kj}^{-1}(\nu_{D(\gamma)(1)}S)\rangle &= \langle x , \lim f_{kj,A}^{-1}(\nu_{D(\gamma)(1)}S)\rangle + o(1) \\
&=  \lim \langle  f_{kj}^{-1/2} x , f_{kj}^{-1/2}(\nu_{D(\gamma)(1)}S)\rangle + o(1) \\
&\ll \langle x, \nu_{D(\gamma)(1)}S\rangle + o(1) < 0.
\end{align}
It proves that $x\in H_x$.
\end{proof}

For any $x\in \widetilde M$, we choose $S\subset \widetilde M$ the maximal convex open subset such that $D(S)$ is a Euclidean  open ball. By applying the construction to the point $x$, the convex $S$ and an incomplete geodesic $\gamma\subset S$ (which exists by  maximality of $S$) we obtain an half-space $H_x=\lim H_i$ and $D(x)\in H_x$. This choice is now assumed.

\begin{lemma}
Let $I\subset \partial H_x$ be the invisible set from the interior of $H_x$. (Note that $(F\oplus P)|_{D(\gamma)(1)}\subset I$.) 
Then $I$ does not depend on $x\in \widetilde M$ and  $I$ is a  nil-affine subspace.
\end{lemma}

\begin{proof}[Proof of theorem~\ref{thm-partialcomp}]
The nil-affine subspace $I$ is constant and contains every $D(\gamma)(1)$ for any incomplete geodesic $\gamma\subset S$ in the maximal open ball of any $x\in \widetilde M$. We show it implies that $D\colon\widetilde M \to \cN-I$ is a covering map.

Let $\delta\colon \iI \to \cN-I$ be a path. Choose $x\in \widetilde M$ such that $D(x)=\delta(0)$. We need to prove that $\delta$ can be lifted to a path in $\widetilde M$ based at $x$. We can assume that $\delta$ can be lifted for $t<1$. We show that is can be lifted at $t=1$. Since $\delta(1)\not \in I$, there exists a point $\delta(s)$ with $s<1$ such that the maximal  open ball based at $\delta(s)$ and avoiding $I$ contains $\delta(t)$ for $s\leq t \leq 1$. But then for the lift at time $s$, the corresponding open ball $S$ is convex and allows to lift $\delta$ for $s\leq t \leq 1$.
\end{proof}

To be more precise on the nature of $I$, we could use a
discreteness argument very close to what Matsumoto~\cite{Matsumoto} proposed (see also~\cite[end of sec. 4]{Ale} and \cite{Ale3}). It shows that in fact $I$ is reduced to $P$ if $F=\{0\}$. In particular, generalizations of Fried theorem~\cite{Fried,Miner,Ale} are proven: if $F=P=\{0\}$ then the developing map is a cover onto the complement of a point.

\section{Reducibility of incomplete manifolds}\label{sec-5}

Before proving the theorems~\ref{thm-markus} and \ref{thm-auto} on closed manifold with rank one ray structures, we examine the invisible subspace $I$. We will prove that $I\subset N$ is in fact a subgroup. This analysis does not depend on the fact that $\mathfrak n$ is two-step nilpotent.

Let $(G_1,\cN)$ be a rank one ray geometry. Let $M$ be a closed incomplete manifold. Let $D\colon\widetilde M \to \cN-I$ be its developing map, with $I$ given by theorem~\ref{thm-partialcomp}. 

We choose a  base point in $I$ that is also a asymptotic fixed point of a Fried dynamic. That is to say, $0\in I$ is the asymptotic fixed point of $T_{ji}$ for $i>0$ fixed and $j\to \infty$. 

We know that $I \supset (P\oplus F)$.
We start by describing the rest of $I$ in $\mathfrak n$, and notably in $E$.

Note that $[E,E]\subset E$. So $E$ is a nilpotent subalgebra of $\mathfrak n$.
Define a decomposition
\begin{equation}
E = E_1\oplus \dots \oplus E_q
\end{equation}
where each $E_i$ is the subspace of $E$ on which $A_1$ (that has rank one) acts by homotheties with the degree $d_i$. Order the degrees such that $d_1<d_2<\dots<d_q$. Observe that since $A_1$ acts by automorphisms, we have $[E_i,E_j]\subset E_{i+j}$.

It gives a graduation of the nilpotent Lie algebra $E$, and we can consider the sequence of ideals
\begin{equation}
E = E^1\supset E^2\supset \dots \supset E^{q+1}=\{0\}
\end{equation}
where
\begin{equation}
E^i = E_i\oplus E_{i+1}\oplus \dots\oplus E_q
\end{equation}
and it verifies $[E,E^i]\subset E^{i+1}$. With lemma~\ref{lem-decomp} in sight we show the following.

\begin{lemma}
The linear subspace $I$ is decomposed by:
\begin{equation}
I = P \oplus F \oplus \bigoplus E_i\cap I.
 \end{equation}
\end{lemma}
\begin{proof}
Since we know that $P\oplus F \subset I$, we only need to show that if $\Delta\subset I$ is a vectorial line in $E$ that has a non vanishing coordinate on $E_i$ for $i$ minimal, then there exists a vectorial line $\Delta'$ of $I$ fully contained in $E_i$.

Because we chose for base point the asymptotic fixed point of $T_{ji}$, we have that $T_{ji}(x)=c_{ji}+f_{ji}(x)$ with $c_{ji}\to 0$. So $T_{ji}(\Delta)$ tends to the same line as $f_{ji}(\Delta)$. Let $\Delta'$ be this limit. It is again in $I\cap E$ since $f_{ji}$ preserves $E$. It is completely contained in $E_i$ since it is the minimal degree of $f_{ji}$ that acts on $\Delta$.

Indeed, assume for example that $i=1$. Then $f_{ji}(\Delta)= f_{ji}(\R x) = \R(\beta^{d_1} x_1,\dots, \beta^d_q x_n)$. By normalizing by $\beta^{d_1}$ we get $\R(x_1,\dots, \beta^{d_j-d_1}x_{l}, \dots, \beta^{d_q-d_1}x_n)$ and every coordinate that is after $E_i$ tends to zero.
\end{proof}

\begin{lemma}
The  subspace $I$ is a subalgebra of $\mathfrak n$.
\end{lemma}
\begin{proof}
We already know that $P\oplus F$ is a subgroup since $[P\oplus F,P\oplus F]\subset P\oplus F$. To prove that $I$ is a subgroup, we show that $[E\cap I,P\oplus F]\subset I$ and that $[E\cap I,E\cap I]\subset I$. Since $I = P\oplus F \oplus (E \cap I)$ (by the preceding lemma) it implies $[I,I]\subset I$ and therefore that $I$ is a subalgebra.

The preceding lemma shows that $E\cap I$ is even the direct sum of every $E_m\cap I$. To prove $[E\cap I,P\oplus F]\subset I$ we let $x_m \in E_m\cap I$ and $e_k\in P\oplus F$ (an element of the basis $(e_1,\dots,e_n)$)  and we show that $[x_m,e_k]\in I$. 

We let $\epsilon>0$ and consider the point $\epsilon x_m \in I$. It gives a Fried dynamic $T_{ji}$ centered asymptotically at a point in $F|_{\epsilon x_m}$. By consequence, without changing the base point of $N$, we use the decomposition $(E\oplus P)\oplus F$ at the origin and we have:
\begin{equation}
T_{ji}(x) = c_{E\oplus P} + c_{ F} + f_{ji}(x)
\end{equation}
where $c_{E\oplus P}\to \epsilon x_m$ by lemma~\ref{lem-convQ}. Now let
$x = -f_{ji}^{-1}(c_{F}) + f_{ji}^{-1}(\epsilon e_k)$. Then $x\in I$ because $f_{ji}$ preserves $P\oplus F$ and $P\oplus F$ is a subgroup. Therefore $T_{ji}(x)\in I$ and
\begin{equation}
T_{ji}(x) = c_{E\oplus P} + \epsilon e_k \to \epsilon x_m + \epsilon e_k\in I.
\end{equation}
In the Lie algebra, 
\begin{equation}
\ln (\epsilon x_m + \epsilon e_k) = \epsilon x_m  + \epsilon e_k + \epsilon^2 \frac 12 [x,e_k] + o(\epsilon^2)
\end{equation}
and by consequence
\begin{equation}
\lim_{\epsilon \to 0}\frac 1{\epsilon^2}2\left(\ln (\epsilon x_m + \epsilon e_k)  - \epsilon x_m -\epsilon e_k \right) = [x_m,e_k].
\end{equation}
But for any $\epsilon$, $\ln (\epsilon x_m + \epsilon e_k)  - \epsilon x_i -\epsilon e_k$ is  in $I$. Hence $[x_m,e_k]\in I$.

Therefore we have proven that $[E\cap I, P\oplus F]\subset I$. It implies $[F,I]\subset I$. 

To prove $[E\cap I,E\cap I]\subset I$ we repeat the same argument but with this time $x_k\in E_k\cap I$ instead of $e_k\in P\oplus F$. We only need to justify that the point $x=-f_{ji}^{-1}(c_{F}) + f_{ji}^{-1}(\epsilon x_k)$ is again in $I$. By invariance under $f_{ji}$, we prove that $F + E_k\cap I\subset I$.

For any point $\exp(p)\in I$ we have $\exp(p)+F\subset I$. Indeed, let $\exp(p)$ be a $D(\gamma)(1)$, then $F|_p\subset I$. Also for any $\exp(q)\in F$ and any $\exp(p)\in I$, since $\Ad(\exp)=\exp(\ad)$ in $\mathfrak n$:
\begin{equation}
\exp(q)+\exp(p) = \exp\left( p + [q,p] +\frac 1{2!} [q,[q,p]] +\dots\right)  + \exp(q).
\end{equation}
Every bracket belongs to $I$ since $[F,I]\subset I$. So $\exp( p + [q,p] +\frac 1{2!} [q,[q,p]] +\dots) $ belongs to $I$ and therefore $\exp(q)+\exp(p)$ belongs to $I$ since $I + \exp(q)$ does.
By consequence, $F + E_k\cap I\subset I$.
\end{proof}

\begin{lemma}\label{lem-ssespace}
Under the same hypotheses, if $T\in \Gamma$ is any holonomy transformation with $T(x)=c+f(x)$. Then $f(I)=I$.
\end{lemma}
\begin{proof}
The holonomy transformation must preserve $I$. Furthermore, since the base point is taken in $I$, we have $c+f(0)=c\in I$. So $f(I) = -c+I$ but $-c+I=I$ since $I$ is a subgroup.
\end{proof}

\begin{lemma}\label{lem-red-decomp}
Under the same hypotheses,  there exists a decomposition $\mathfrak n = I\oplus V$  such that  $I$ and $V$ are both invariant under the linear holonomy of $M$. Furthermore, lemma~\ref{lem-decomp} applies and $N=\exp(I)\exp(V)$ by a diffeomorphism.
\end{lemma}

\begin{proof}
On each $E_i$, the linear group $KA_1$ acts by similarity: $A_1$ acts by homotheties and since $K$ centralizes $A_1$, it preserves $E_i$.

By the choice of the Fried dynamic $T_{ji}$ with asymptotic fixed point $p\in I$, we know that $(P\oplus F)|_p\subset I$. So the linear subspaces $L_j$ such that the degrees are $d_j\leq 0$ are all contained in $I$. To construct $V$, we only need to inspect the subspaces on which the degrees are $d_j>0$ (it corresponds to the subspaces of $E|_p$). 

Since $A_1$ acts by homotheties on $E_i$, we can consider a scalar product on $E_i$ that is invariant under $KA_1$. Let $I_i = I\cap E_i$ and $V_i = I_i^\bot$. By lemma~\ref{lem-ssespace} the linear holonomy preserves $I_i$ and therefore it  preserves its orthogonal $V_i$. 
We let $V = V_1\oplus \dots\oplus V_k$. It is stable under the linear holonomy and gives $V\oplus I = \mathfrak n$.

It lasts to show that we have again a decomposition $N=\exp (I)\exp V$. First, in $E$ we can decompose by $I\cap E\oplus V$. Indeed, each $E_i$ is generated by $E_i\cap I$ and $E_i\cap V$. So lemma~\ref{lem-decomp} applies and for any $x_E\in E$, there is a unique decomposition $x_E = x_{E\cap I} + x_V$. 

Now decompose by $(P\oplus F) \oplus E$. We get that for any $x\in N$, $x = x_{P\oplus F}+x_E = x_{P\oplus F}+x_{E\cap I}+x_V$. Since $I$ is a subgroup and $P\oplus F\subset I$, we have $x_{P\oplus F} + x_{E\cap I} = x_I\in I$.

It lasts to show that $x= x_I+x_V$ is unique. The point $x_V$ is unique so $x_I'+x_V=x_I+x_V$ implies $x_I'=x_I$.
\end{proof}

\subsection{Markus conjecture}
We now address the link  with Markus conjecture~\cite{Markus}. This conjectures states that (in real affine geometry) \emph{closed manifolds with parallel volume are complete}.

\begin{definition}
A ray geometry $(G,\cN)$ has \emph{parallel volume} if $G$ preserves the volume form $\dd x_1\wedge\dots\wedge \dd x_n$, with $(x_1,\dots,x_n)$ the dual coordinate functions of the basis $(e_1,\dots,e_n)$ given by the geometry.
\end{definition}

If we were in real affine geometry, theorem~\ref{thm-partialcomp} would show that closed manifolds with parallel volume are complete. Indeed, the fact that it preserves $I$ implies that the holonomy is reducible. But by Goldman-Hirsch~\cite{GH}, the holonomy of such a closed manifold can never be reducible.

This argument that we used in~\cite{Ale3} can not be applied here since Goldman-Hirsch construction relies on the commutative law of $\R^n$. Instead we prove theorem~\ref{thm-markus} by an argument similar to what Fried~\cite{Fried} employed.

\begin{theorem}\label{thm-markus}
\thmmarkus
\end{theorem}

\begin{proof}
Consider a $(G,\cN)$-manifold that has its developing map covering $\cN-I$. We show a contradiction.

Denote $H\subset G$ the stabilizer of $I$. Each point of $I$ provides a Fried dynamic.

Consider the decomposition $\mathfrak n = I\oplus V$ given by the preceding lemma.
By lemma~\ref{lem-decomp}, for any $x\in N$ there exists a unique decomposition $x = x_I+ x_V$ with $x_I\in \exp(I)$ and $x_V\in \exp(V)$. On $N-I$ we define the vector field $X$ by:
\begin{equation}
X(x) = (L_x)_*\ln(x_V)
\end{equation}

The vector field $X$ is $H$-invariant. Indeed, if $T\in H$ then $T(x)= c_I + f(x)$ with $c_I\in \exp(I)$ since $I$ is stable. So
\begin{equation}
X(T(x)) = X(c_I + f(x_I) + f(x_V)) = (L_{T(x)})_*f(\ln(x_V)) = T_*X(x)
\end{equation}
since $c_I+f(x_I)$ must belong  to $\exp(I)$ and $V$ is stable under the linear holonomy.

\vskip10pt
Now, similarly to the last argument of Fried~\cite{Fried}, we show that $X$ can't preserve the parallel volume, in contradiction with the fact that $M$ is closed.
Recall that we have the parallel volume in the coordinates $(x_1,\dots,x_n)$ associated to the $(e_1,\dots,e_n)$:
\begin{equation}
\vol_N = \dd x_1\wedge\dots\wedge \dd x_n.
\end{equation}

By hypothesis, $\vol_N$ is preserved by $G$, and so it is also by $H$. The pulled-back $D^*\vol_N$ defines a volume form on $\widetilde M$ that is invariant under $\pi_1(M)$. So it gives a volume form on the closed manifold $M$. Any flow of any vector field on $M$ must preserve the volume of  $M$ since it is closed.

But $X$ is $H$-invariant, so by the same procedure, provides a vector field $Y$ on $M$. The flow of $Y$ corresponds to the flow of $X$. Choose $\cup U_i$ a cover of $M$ by open sets so that there exists corresponding lifts $\cup \widetilde U_i$ in $\widetilde M$ that are each diffeomorphically sent to $V_i$ in $N$ by $D\colon \widetilde M \to N$. Choose a partition of the unity $\{\rho_i\}$ in $M$ associated to $\{U_i\}$. Denote $R_t$ the flow of $X$ in $N$ and $\Phi_t$ the flow of $Y$ in $M$. Then the conservation of the volume in $M$ implies that for any $t\in \R$:
\begin{align}
\int_M \Phi_t^*\vol_M &= \int_M \vol_M \\
\iff \sum_{j} \int_{U_j} \rho_j \Phi_t^* \vol _M& = \sum_j \int_{U_j} \rho_j \vol_M \\
\iff \sum_j \int_{V_j} \rho_j R_t^*\vol_N &= \sum_j \int_{V_j}\rho_j \vol_N\\
\iff \sum_j \int_{V_j}\rho_j \left(R_t^*\vol_N - \vol_N\right) &= 0.
\end{align}

But now, observe that
\begin{equation}
R_t(x_I+x_V) = x_I + e^t x_V.
\end{equation}
Indeed, for any fixed $x$, $R_t(x)$ must be the solution to the differential equation $R_t(x)^*\omega_N = \omega_N(X(R_t(x)))$, where $\omega_N$ is the Maurer-Cartan form of $N$. By left-invariance and the notations,
\begin{equation}
R_t(x)^*\omega_N = \exp(e^t \ln(x_V))^*\omega_N.
\end{equation}
Note that the exponential map $f(t)=\exp(tv)$ is precisely the solution to the equation $f^*\omega_N=v$ with $f(0)=e$.  Let $t\in \R$ and $\epsilon \to 0$,
by the Baker-Campbell-Hausdorff formula and $[x_V,x_V]=0$,
\begin{align}
\exp(e^{t+\epsilon}\ln(x_V)) &= \exp(e^t(1+\epsilon+o(\epsilon))\ln(x_V)) \\
&= \exp(e^t\ln(x_V))\exp((\epsilon+o(\epsilon))e^t\ln(x_V))
\end{align}
and by left-invariance, we obtain
\begin{equation}
\exp(e^t \ln(x_V))^*\omega_N = e^t\ln(x_V) = \omega_N(X(R_t(x))).
\end{equation}

Therefore
\begin{equation}
R_t^*\vol_N  = e^{t\dim V} \vol_N
\end{equation}
and it contradicts the preservation of the volume since for $t>0$ large enough $(R_t^*\vol_N - \vol_N)$ is a positive function.
\end{proof}

\subsection{The automorphism group}

It is a vague conjecture~\cite{Gromov} that geometric manifolds with large automorphism groups  should be classifiable.

\begin{definition}
Let $M$ be a $(G,X)$-manifold. An \emph{automorphism} $f\colon M \to M$ is a diffeomorphism such that if $\widetilde f\colon\widetilde M \to \widetilde M$ is any lift  then there exists a unique $\chi(\widetilde f)\in N_G(\Gamma)$ in the normalizer of the holonomy group, such that $D(\widetilde f(x)) = \chi(\widetilde f)D(x)$.

The group $\Aut(M)$ of its automorphism is sent by $\chi$ into a subgroup of $\Gamma \backslash N_G(\Gamma)$.
\end{definition}

By unicity of $\chi(\widetilde f)$, if $\widetilde f_1$ and $\widetilde f_2$ are two lifts of $f$ then $\widetilde f_2 = g\widetilde f_1$ for $g\in \pi_1(M)$ and therefore $\chi(\widetilde f_2) = \rho(g)\chi(\widetilde f_1)$.

The hypothesis that $\chi(\widetilde f)$ normalizes $\Gamma$ follows from the fact that if $\widetilde f$ lifts $f$ then it must preserves the fibers $\pi_1(M)$ of $\widetilde M$ and therefore 
\begin{equation}
\chi(\widetilde f) \Gamma \cdot D(x) = D(\widetilde f(\pi_1(M) \cdot x)) = D(\pi_1(M)\cdot \widetilde f(x)) = \Gamma\cdot  \chi(\widetilde f)D(x).
\end{equation}

\vskip10pt
If the rank one ray geometry $(G_1,\cN)$
 is a subgeometry of a Levi geometry of a rank one parabolic geometry, then by Ferrand-Obata~\cite{Ferrand,Obata}, Schoen-Webster~\cite{Schoen,Webster} and Frances~\cite{Frances} we would get that if the automorphism group acts non properly then the manifold is complete.

We show that this phenomenon is only dependent on the fact that the ray geometry has rank one on an at most two-step nilpotent space. It is indeed the case if it comes from a rank one parabolic geometry.

\begin{theorem}\label{thm-auto}
\thmauto
\end{theorem}

\begin{proof}
Assume that $M$ is not complete, then by theorem~\ref{thm-partialcomp}, we get that $D\colon\widetilde M \to \cN-I$ is a cover. Therefore the holonomy $\Gamma$ preserves $I$ and so does the normalizer $N(\Gamma)$. Both must be subgroups of $I\rtimes KA_1$, the subgroup of $G_1$ stabilizing $I$.
We show that $\Aut(M)$ must act properly.

Choose a base point $p\in I$ which is the asymptotic fixed point of a Fried dynamic and consider the decomposition $I\oplus V=\mathfrak n$ given by lemma~\ref{lem-red-decomp}. We identify $I$ with its linear subspace in $\mathfrak n$. 
In rank one, $I$ contains $(P\oplus F)|_p$ and therefore $V$  is contained in $E|_p$.
By lemma~\ref{lem-decomp} we get that any $x\in N$ has a unique decomposition $x = x_I+ x_V$ with $x_I\in\exp(I)$ and $x_V=\exp(V)$.

\vskip10pt
Assume that $\Aut(M)$ acts non properly. It is equivalent to the existence of $x_n\to x$ in $\widetilde M$, of $g_n\in N(\Gamma)$ with $\Gamma g_n$ escaping every compact of $\Gamma\backslash N(\Gamma)$, and of $y\in\widetilde M$ such that $\Gamma g_n  D(x_n) \to \Gamma D(y)$. 
 
 Write
$g_n(x) = c_n + f_n(x)$ with $c_n\in I$ and $f_n\in KA_1$.
The Fried dynamic $T_{ji}\in \Gamma$ considered can be written $T_{ji}{x} = c_{ji}+f_{ji}(x)$. Note that
\begin{equation}
T_{ji}g_n(x) = c_{ji} + f_{ji}(c_n) + f_{ji}f_n(x).
\end{equation}
It shows that we can assume $f_{ji}f_n$ bounded in $KA_1$ since $A_1$ has rank one (up to exchange $T_{ji}$ with $T_{ji}^{-1}$).

Hence there exists  $\gamma_n\in \Gamma$ such that $\gamma_n g_n=  h_n$ has its $KA_1$-factor that converges. The convergence  $\Gamma g_n D(x_n)\to \Gamma D(y)$ says that there exists  $\eta_n\in \Gamma$ such that $\eta_nh_n D(x_n) \to D(y)$.
 
Write 
$\eta_n(x) = b_n + q_n(x)$ and $h_n(x) = c_n + r_n(x)$ we have by construction $r_n\to r$ and
 \begin{equation}
 (\eta_n h_n)(D(x_n)) = b_n + q_n(c_n) + q_nr_n(D(x_n)) \to D(y).
 \end{equation}
 We again have a decomposition
 \begin{equation}
  (\eta_n h_n)(D(x_n)) = ( (\eta_n h_n)(D(x_n)))_I +  ((\eta_n h_n)(D(x_n)))_V.
 \end{equation}
 The term $b_n+q_n(c_n)$ must belong to $I$ since $c_n\in I$. So the term $((\eta_nh_n)(D(x_n)))_V$ tending to  $D(y)_V$ is determined by $(q_nr_n(D(x_n)))_V$.
But since $D(y)_V\neq 0$ and $V\subset E$, it shows that $q_n$ itself must converge to $q\in KA_1$.

The term $((\eta_nh_n)(D(x_n)))_I$ tending to $D(y)_I$ has the same limit as the $I$-factor of $b_n+q_n(c_n)+qr(D(x_n))$.
 Since $D(y)_I$ and $D(x)=\lim D(x_{n})$ are both  finite, $b_n+q(c_n)$ must converge.
 
 Therefore $\eta_nh_n$ converges in $I\rtimes KA_1$, contradicting the fact that $\Gamma g_n$ escapes every compact of $\Gamma\backslash N(\Gamma)$.
\end{proof}

\paragraph{In higher rank}As in ~\cite{Ale3}, we can give a relatively generic example that in rank $r>1$ this phenomenon is no longer true. Let $(x,y,z)$ be coordinates of the Heisenberg group in dimension three and consider the rank two diagonal group be given by $(\beta_1 x, \beta_2 y, \beta_1\beta_2 z)$. Then consider the radiant manifold given by $\cN-\{0\}$ quotiented by the subgroup generated by $(2x,2y,4z)$. Then let $f(x,y,z) = (x/2,y,z/2)$ and $p=(x,1,0)$. Then $f^n(p)\to (0,1,0)$ and it corresponds in $M$ to an automorphism acting non properly.

\printbibliography
\end{document}